\documentclass[12pt]{article}
\usepackage{amsthm}
\newtheorem{theorem}{Theorem}

\newtheorem{obs}{Observation}
\newtheorem{corollary}{Corollary}
\newtheorem*{obs*}{Observation}

\newtheorem{conj}{Conjecture}

\newtheorem*{def*}{Definition}
\newtheorem*{lemma*}{Lemma}
\usepackage{amssymb}
\usepackage{mathtools}
\usepackage{graphicx}
\usepackage{caption}
\usepackage{enumitem}

\def \nmr {\begin{enumerate}}
\def \enmr {\end{enumerate}}

\def \tmz {\begin{itemize}}
\def \etmz {\end{itemize}}

\newcommand{\Proof}{\noindent\textbf{Proof. }}
\newcommand{\smallqed}{{\tiny ($\Box$)}}

\newcommand{\w}{{\rm s}}

\title{Domination number of graphs with minimum degree five\thanks{Research supported by the Slovenian Research Agency under the project N1-0108}}

\begin{document}

\author{Csilla Bujt\'as}
\date{\empty}
\maketitle

\vspace{-5ex}

 \begin{center}
 	\small Faculty of Mathematics and Physics, University of Ljubljana\\
 	Ljubljana, Slovenia\\
 	\texttt{csilla.bujtas@fmf.uni-lj.si}
 \end{center}

\begin{abstract}
We prove that for every graph $G$ on $n$ vertices and with minimum degree five, the domination number $\gamma(G)$ cannot exceed $n/3$. The proof combines an algorithmic approach and the discharging method. Using the same technique, we provide a shorter proof for the known upper bound $4n/11$ on the domination number of graphs of minimum degree four.
\end{abstract}


\section{Introduction}
In this paper we study the minimum dominating sets in graphs of given order $n$ and minimum degree $\delta$. For the case of $\delta= 5$, we improve the previous best upper bound $0.344\,n$ by proving that the domination number $\gamma$ is at most $n/3$. For graphs of $\delta=4$, the relation $\gamma\le 4n/11$ was proved by Sohn and  Xudong \cite{sohn-2009} in 2009. Using a different approach, we provide a simpler proof for this theorem.

\paragraph{Standard definitions.}
In a simple graph $G$, the vertex set is denoted by $V(G)$ and the edge set by $E(G)$. For a vertex $v \in V(G)$, its \emph{closed neighborhood} $N[v]$ contains $v$ and its neighbors. For a set $S \subseteq V(G)$, we use the analogous notation $N[S]= \bigcup_{v \in S} N[v]$. The \emph{degree} of a vertex $v$ is denoted by $d(v)$, while $\delta(G)$ and $\Delta(G)$, respectively, stand for the \emph{minimum and maximum vertex degree} in $G$. A set $D \subseteq V(G)$ is a \emph{dominating set}  if $N[D]=V(G)$. The minimum cardinality of a dominating set is the \emph{domination number} $\gamma(G)$ of the graph. An earlier general survey on domination theory is~\cite{haynes-1998}, while two new directions were initiated recently  in~\cite{bresar-2010} and~\cite{bresar-2014}.

\paragraph{General upper bounds on $\gamma(G)$ in terms of the order and minimum degree.}
The first general upper bound on $\gamma(G)$ in terms of the order $n$ and the minimum degree $\delta$ was given by Arnautov~\cite{arnautov-1974} and, independently, by Payan~\cite{payan-1975}:
\begin{equation}
\label{eq:arnautov}
\gamma(G) \le \frac{n}{\delta +1}\sum_{j=1}^{\delta+1}\frac{1}{j}\,.
\end{equation}
We remark that a bit stronger general results were later published by Clark et al.~\cite{CSSF} and Bir\'o et al.~\cite{BCDS}. On the other hand, already~\eqref{eq:arnautov} implies the upper bound
\begin{equation}
\label{eq:alon}
\gamma(G) \le n \left( \frac{1 + \ln (\delta + 1)}{\delta +1}\right)\, .
\end{equation}
It was proved by Alon~\cite{alon-1990} that~\eqref{eq:alon} is asymptotically sharp when $\delta \to \infty$.

\paragraph{Upper bounds for graphs of small minimum degrees.} There are several ways to show that $\gamma(G) \le n/2$ holds if $\delta(G) = 1$ (see~\cite{ore-1962} for the first proof). Blank~\cite{blank-1973}, and later independently McCuaig and Shepherd~\cite{mccuaig-1989} proved that $\gamma(G)\le 2n/5$ is true if $G$ is connected, $\delta(G)= 2$, and $n \ge 8$.\footnote{There are seven small graphs, the  cycle $C_4$  and six graphs with $n=7$ and  $\delta=2$, which do not satisfy $\gamma(G)\le 2n/5$.} For graphs $G$ with $\delta(G)=3$, Reed~\cite{reed-1996} proved the famous result that $\gamma(G) \le 3n/8$. He also presented a connected cubic graph on $8$ vertices for which the upper bound is tight. 

In the same paper~\cite{reed-1996}, Reed provided the conjecture that the upper bound can be improved to $\lceil n/3 \rceil$ once the connected cubic graph has an appropriately large order. It  was disproved by Kostochka and Stodolsky~\cite{KS-2005} by constructing an infinite sequence of connected cubic graphs such that all of them have $\gamma(G) \ge (\frac{1}{3}+\frac{1}{69})\, n$. Later, in~\cite{kostochka-2009}, the same authors proved that $\gamma(G) \le  \frac{4}{11} n = (\frac{1}{3}+\frac{1}{33})\, n$ holds for every connected cubic graph of order $n >8$. However, it seems a challenging and difficult problem to close the small gap between $\frac{1}{3}+\frac{1}{69}$ and $\frac{1}{3}+\frac{1}{33}$.

For graphs of minimum degree $4$, the best known upper bound is $\gamma(G) \le  \frac{4}{11}\, n$ that was established by Sohn and Xudong~\cite{sohn-2009}. For the case of $\delta(G)=5$, Xing, Sun, and Chen~\cite{XSC} proved $\gamma(G) \le  \frac{5}{14} \, n$ which was improved to $\gamma(G) \le  \frac{2671}{7766}\, n < 0.344 \,n$ by the authors of~\cite{BK-2016}. It was also shown in~\cite{BK-2016} that for graphs of minimum degree $6$, the domination number is strictly smaller than $n/3$. 
Note that similar upper bounds involving the girth and other parameters of the graph can be found in many papers, e.g.\ in ~\cite{DJLMR, HSY, kral-2012,lo-2008}, while results for plane triangulations and maximal outerplanar graphs were established in \cite{KP-2010} and \cite{CW-2013}. 

\paragraph{Our approach.}
In the seminal paper~\cite{reed-1996} of Reed, the upper bound $3n/8$ was proved by considering a vertex-disjoint path cover with specific properties. Later, the same method (with updated conditions and thorough analysis) was used in~\cite{kostochka-2009, sohn-2009, XSC} to establish results on cubic graphs and on graphs of minimum degree $4$ and $5$. In~\cite{BK-2016}, we introduced a different algorithmic method that resulted in improvement for all cases with $5 \le \delta \le 50$. Here, we combine the latter approach with a discharging process. This allows us to prove that already graphs of minimum degree $5$ satisfy $\gamma(G) \le n/3$.

\paragraph{Residual graph.}
Given a graph $G$ and a set $D \subseteq V(G)$, the \emph{residual graph} $G_D$ is obtained from $G$ by assigning colors to the vertices and deleting some edges according to the following definitions:
\begin{itemize}
\item A vertex $v$ is \emph{white} if $v \notin N[D]$.
\item  A vertex $v$ is \emph{blue} if $v \in N[D]$ and $N[v] \not\subseteq N[D]$.
\item A vertex $v$ is \emph{red} if $N[v] \subseteq N[D]$.
\item $G_D$ contains only those edges from $G$ that are incident to at least one white vertex.
\end{itemize}
In $G_D$, we refer to the set of white, blue, and red vertices, respectively, by the notations $W$, $B$, and $R$. It is clear by definitions that $D \subseteq R$ and  $W \cup B \cup R= V(G)$ hold. The \emph{white-degree} $d_W(v)$ of a vertex $v$ is the number of its white neighbors in $G_D$. Analogously, we sometimes refer to the \emph{blue-degree} $d_B(v)$ of a vertex. The maximum of white-degrees over the sets of white and blue vertices, respectively, are denoted by $\Delta_W(W)$ and $\Delta_W(B)$. 

\begin{obs}
	\label{obs:1} Let $G$ be a graph and $D \subseteq V(G)$. The following statements are true for the residual graph $G_D$.
\begin{itemize}
\item[$(i)$] If\/ $v \in W$, then $G_D$ contains all edges which are incident with $v$ in $G$ and, in particular, $N[v] \cap R= \emptyset$ and $d_W(v)+d_B(v) =d(v)$  hold.
\item[$(ii)$] If\/ $v \in B$, then $d_W(v)= |W \cap N[v]| < d(v)$ and $d_B(v)=0$.
\item[$(iii)$] If\/ $v \in R$, then $v$ is an isolated vertex in $G_D$.  
\item[$(iv)$] If\/ $\delta(G) = d$ and $v$ is a white vertex with $d_W(v)= \ell < d$, then $d_B(v) \ge d-\ell $ holds in $G_D$.
\item[$(v)$] $D$ is a dominating set of $G$ if and only if\/ $R = V(G)$ (or equivalently, $W = \emptyset$) in $G_D$.
\item[$(vi)$] If\/ $D \subseteq D' \subseteq V(G)$ and a vertex $v$ is red in $G_D$, it remains red in $G_{D'}$; if $v$ is blue in $G_D$, then it is either blue or red in $G_{D'}$.
\end{itemize}
\end{obs}

\paragraph{Structure of the paper.}
In the next section we prove the improved upper bound $n/3$ on the domination number of graphs with minimum degree $5$. In Section~\ref{sec:deg-4} we consider graphs of minumum degree $4$ and show an alternative proof for the theorem $\gamma \le 4n/11$. 

\section{Graphs of minimum degree $5$}
\label{sec:deg-5}

\begin{theorem}
	\label{thm:delta-5}
 For every graph $G$ on $n$ vertices and with minimum degree $5$, the domination number satisfies $\gamma(G) \le \frac{n}{3}$.
\end{theorem}
\begin{proof}
Consider a graph $G$ and a subset $D$ of the vertex set $V=V(G)$. Let $W$, $B$, and $R$ denote the set of white, blue, and red vertices  respectively,  in the residual graph $G_D$. Further, for the sets of blue vertices that have at least $5$ white neighbors, or exactly $4$, $3$, $2$, $1$ white neighbors, we use the notations $B_5$, $B_4$, $B_3$, $B_2$, and $B_1$  respectively. A vertex is a blue leaf if it belongs to $B_1$. In the proof, a residual graph $G_D$ is associated with the following value:
$$f(G_D)= 35|W|+23|B_5|+21|B_4| +19|B_3| +17|B_2| +14|B_1|.$$
By Observation~\ref{obs:1} $(v)$, $f(G_D)$ equals zero if and only if $D$ is a dominating set in $G$. If $G$ and $D$ are fixed and $A$ is a subset of $V\setminus D$, we define 
$$\w(A)=f(G_D)-f(G_{D\cup A})$$
that is the decrease in the value of $f$ when $D$ is extended by the vertices of $A$.  We define the following property for $G_D$:
 \paragraph{Property 1.} There exists a nonempty set $A \subseteq V\setminus D$ such that $\w(A) \ge 105\,|A|$.
 \medskip 
 
 Our goal is to prove that every graph $G$ with $\delta(G)=5$ and every $D \subseteq V$ with $f(G_D)>0$ satisfy Property 1. Once we do it, Theorem~\ref{thm:delta-5} will follow easily. In the continuation, we suppose that a graph $G$ with minimum degree $5$ and a set $D$ with $f(G_D)>0$ do not satisfy Property 1 and prove, by a series of claims, that this assumption leads to a contradiction.
 
\paragraph{Claim A.}  In $G_D$, every white vertex $v$ has at most two white neighbors, and every blue vertex $u$ has at most three white neighbors.

\Proof First suppose that there is vertex $v \in W$ with $d_W (v) \ge 6$. Choosing $A=\{v\}$, the white vertex $v$ becomes red in $G_{D\cup A}$ that decreases $f$ by $35$. The white neighbors of $v$ become blue or red which decreases $f$ by at least $6\cdot (35-23)$. Hence, we have $\w(A)\ge 35+72=107 >105\,|A|$ complying with Property 1. This contradicts our assumption on $G_D$ and implies that $\Delta_W(W) \le 5$.  

Now, suppose that $\Delta_W(W) = 5$ in $G_D$. Let $v$ be a white vertex with $d_W(v)=5$ and consider $A=\{v\}$. In $G_{D\cup A}$, the vertex $v$ becomes red and its white neighbors become blue (or red). Since each neighbor $u$ had at most $5$ white neighbors in $G_D$ and at least one of them, namely $v$, becomes red, $u$ may have at most $4$ white neighbors in $G_{D\cup A}$. Therefore, $\w(A) \ge 35+ 5 \cdot (35-21) = 105\, |A|$  holds which is a contradiction again. 

If $\Delta_W(W) \le 4$ and $\Delta_W(B) \ge 6$, let $v$ be a blue vertex with $d_W(v) \ge 6$ and define $A=\{v\}$ again. In $G_D$, the vertex $v$ belongs to $B_5$, while we have $v\in R$ in $G_{D\cup A}$ which causes a decrease of $23$ in the value of $f$.  Each white neighbor $u$ of $v$ has at most four white neighbors in $G_D$ and, therefore, $u \in B_4 \cup B_3 \cup B_2 \cup B_1 \cup R$ in $G_{D\cup A}$. Hence, we have $\w(A) \ge 23+6(35-21)=107 >105\, |A|$, a contradiction to our assumption. Note that in the continuation, where we suppose $\Delta_W(B) \le 5$, if a blue vertex loses $\ell$ white neighbors in a step, it causes a decrease of at least $2\ell$ in the value of $f$.

Assume that $\Delta_W(W) = 4$ and $\Delta_W(B) \le 5$ and let $v$ be a white vertex with $d_W(v) =4$ in $G_D$. Set $A=\{v\}$ and consider the decrease $\w(A)$. As $v$ turns to be red, this contributes by $35$ to $\w(A)$. The four white neighbors become blue (or red) and each of them has at most $3$ white neighbors  in $G_{D \cup A}$. Hence, the contribution to $\w(A)$ is at least $4(35-19)$. Further, we have $d_W(u) \le 4$ for each white vertex $u$ from $N[v]$. This implies, by Observation~\ref{obs:1} $(iv)$, that $u$ has at least one blue neighbor in $G_D$ the white-degree of which is smaller in $G_{D \cup A}$ than in $G_D$. Even if some blue vertices from $N[N[v]]$ have more than one neighbor from $N[v]$, it remains true that the sum of the white-degrees over $B \cap N[N[v]]$ decreases by at least $d_W(v)+1 = 5$. We may conclude $\w(A) \ge 35+ 4(35-19) + 5\cdot 2= 109 >105\, |A|$. 

Assume that $\Delta_W(W) \le 3$ and $\Delta_W(B) = 5$ hold in $G_D$ and $v$ is a blue vertex with $d_W(v) =5$. Let $A=\{v\}$ and consider the decrease $\w(A)$. 
Since $v$ belongs to $B_5$ in $G_D$ and to $R$ in $G_{D \cup A}$, this change contributes by $23$ to $\w(A)$. The five white neighbors of $u$ become blue or red and belong to $ B_3 \cup B_2 \cup B_1 \cup R$ in $G_{D\cup A}$. The contribution to $\w(A)$ is not smaller than $5(35-19)$. By Observation~\ref{obs:1} $(iv)$ and by $\Delta_W(W) \le 3$, each white vertex has at least two blue neighbors in $G_D$. That is, each white neighbor has at least one blue neighbor that is different from $v$. As the five white vertices from $N(v)$ turn blue (or red) in $G_{D \cup A}$, the sum of the white-degrees over $B\cap (N[N[v]]\setminus \{v\})$ decreases by at least $5$. We infer that $\w(A) \ge 23+ 5(35-19)+5\cdot 2=113 >105\, |A|$ which is a contradiction again. 

The next case which we consider is $\Delta_W(W) = 3$ and $\Delta_W(B) \le 4$. Let $v$ be a white vertex with $d_W(v)=3$ and estimate the value of $\w(A)$ for $A=\{v\}$. When $D$ is replaced by $D\cup A$, vertex $v$ is recolored red, the three white neighbors of $v$ become blue or red and belong to $B_2 \cup B_1 \cup R$ in $G_{D\cup A}$. Additionally, each of the three white neighbors and also $v$ itself has at least two blue neighbors. The decrease in their white-degrees contributes to $\w(A)$ by at least $4\cdot 2 \cdot 2$. Consequently, we have $\w(A) \ge 35+3(35-17)+16=105\, |A|$ that is a contradiction.

The last case is when $\Delta_W(W) \le 2$ and $\Delta_W(B) =4$. We assume that $v$ is a vertex from $B_4$ in $G_D$. Let $A=\{v\}$ and observe that $v$ is recolored red and the white neighbors of $v$ belong to $B_2 \cup B_1 \cup R$ in $G_{D\cup A}$. Since now we have  $\Delta_W(W) \le 2$ in $G_D$, each white vertex has at least three blue neighbors. Therefore, each white neighbor of $v$ has at least two blue neighbors which are different from $v$. We conclude that $\w(A) \ge 21+4(35-17)+ 4\cdot 2 \cdot 2 =109 >105\, |A|$. This contradiction finishes the proof of Claim A.
\smallqed
\medskip

From now on we may suppose that $\Delta_W(W) \le 2$ and $\Delta_W(B) \le 3$ holds in the counterexample $G_D$. This implies that the graph $G_D[W]$, which is induced by the white vertices of $G_D$, contains only paths and cycles as components. Before performing a discharging, we prove some further properties of $G_D$. 

\paragraph{Claim B.}  In $G_D[W]$, each component is a path $P_1$, $P_2$ or a cycle $C_4$, $C_5$, $C_7$ or $C_{10}$.

\Proof First, suppose that $P_j\colon v_1\dots v_j$ is a path component on $j \ge 3$ vertices in $G_D[W]$. Let us choose $A=\{v_2\}$. In $G_{D\cup A}$ not only $v_2$ but also $v_1$ becomes red, while $v_3$ turns to be either a blue leaf or a red vertex. These changes contribute to $\w(A)$ by at least $2\cdot 35 + (35- 14)$. By Observation~\ref{obs:1} $(iv)$, $v_1$, $v_2$, and $v_3$, respectively, have at least $4$, $3$, $3$ blue neighbors in $G_D$. The decrease in their white-degrees contributes to $\w(A)$ by at least $20$. We may infer that $\w(A) \ge 70+21+20=111 >105\, |A|$, a contradiction to our assumption.

We now prove that no cycle of length $3k$ occurs in $G_D[W]$. Assuming that a cycle $C_{3k}\colon v_1 \dots v_{3k}v_1$ exists, all vertices of it can be dominated by the $k$-element set $A=\{v_3, v_6, \dots , v_{3k}\}$. Then, in $G_{D \cup A}$, all the $3k$ vertices are red and, by Observation~\ref{obs:1} $(iv)$, the sum of the white-degrees of the blue neighbors decreases by at least $3\cdot 3k$. Consequently, we get the contradiction $w(A) \ge 35\cdot 3k+ 2\cdot 9k = 123k > 105\, |A|$. 

Similarly, if we suppose the existence of a cycle $C_{3k+2} \colon v_1 \dots v_{3k+2}v_1$ with $k \ge 2$ and define $A=\{v_3, v_6, \dots , v_{3k}, v_{3k+2}\}$, the set $A$ dominates all vertices. Since $k \ge 2$, the relation $\w(A) \ge 35\cdot (3k+2) + 2\cdot 3\cdot(3k+2) = 123k +82 > 105(k+1) =105\, |A|$ clearly holds and gives the contradiction.

In the last case, consider a cycle $C_{3k+1}\colon v_1 \dots v_{3k+1}v_1$ with $k \ge 4$ and set $A=\{v_3, v_6, \dots , v_{3k}, v_{3k+1}\}$. In $G_{D\cup A}$, every vertex from the cycle is red and, as before, one can prove that $\w(A) \ge 35\cdot (3k+1) + 2\cdot 3\cdot(3k+1) = 123k +41 > 105(k+1) =105\, |A|$. This contradiction finishes the proof of Claim B. \smallqed

\medskip
For $i=0,1,2$, we will use the notation $W_i$ for the set of white vertices having exactly $i$ white neighbors in $G_D$. Note that $W_0$  consists of the vertices of the components of $G_D[W]$ which are isomorphic to $P_1$, while $W_1$ and $W_2$, respectively, contain  the vertices from the $P_2$-components and the cycles of $G_D[W]$.

\paragraph{Claim C.}  No vertex from $B_3$ is adjacent to a vertex from $W_0$ in $G_D$.

\Proof In contrary, suppose that a vertex $v\in B_3$ has a neighbor $u$ from $W_0$. Let $A=\{v\}$ and denote by $u_1$ and $u_2$ the further two white neighbors of $v$. In $G_{D\cup A}$, we have $v,u \in R$ and $u_1,u_2 \in B_2 \cup B_1 \cup R$. This contributes to $\w(A)$ by at least $19+35+2(35-17)=90$. By Observation~\ref{obs:1} $(iv)$, the neighbors $u$, $u_1$ and $u_2$ have, respectively, at least $4$, $2$, $2$ blue neighbors which are different from $v$. As follows, $\w(A) \ge 90+ 2 \cdot 8 = 106 > 105\, |A|$ must be true but this contradicts our assumption on $G_D$. \smallqed
\medskip
 
 We call a vertex from $B_2$ \emph{special}, if it is adjacent to a vertex from $W_0$. 

\paragraph{Claim D.}  No special vertex is adjacent to two vertices from $W_0$. 

\Proof Suppose that a vertex $v \in B_2$ is adjacent to two vertices, say $u_1$ and $u_2$ from $W_0$. Then, we set $A=\{v\}$ and observe that all the three vertices $v$, $u_1$ and $u_2$ are red in $G_{D\cup A}$. By Claim C, all the blue neighbors of $u_1$ and $u_2$ are from $B_2 \cup B_1$ in $G_D$ and, therefore, when the white-degree of these neighbors decreases by $\ell$, the value of $f$ falls by at least $(17-14)\ell= 3\ell$. Since, by Observation~\ref{obs:1} $(iv)$, each of $u_1$ and $u_2$ has at least four blue neighbors, we have $\w(A) \ge 17+2\cdot 35+ 3 \cdot 8= 111>105\, |A|$. This contradiction proves the claim. \smallqed 
  
\paragraph{Claim E.}  No special vertex is adjacent to a vertex from a $C_4$ or $C_7$. 

\Proof Suppose first that a special vertex $v\in B_2$ is adjacent to $u_1$ which is from a $4$-cycle component $C_4\colon u_1u_2u_3u_4u_1$ in $G_D$. The other neighbor of $v$ is  $u_0$ which is from $W_0$. Let $A=\{v, u_3\}$ and observe that all the six vertices $v$, $u_0$, $u_1$, $u_2$, $u_3$ and $u_4$ are red in $G_{D\cup A}$. In $G_D$, the white vertex $u_0$ has at least four blue neighbors which are different from $v$ and, by Claim C, each of them belongs to $B_2 \cup B_1$;  $u_1$ has at least two neighbors from $(B_3 \cup B_2 \cup B_1)\setminus \{v\}$; each of $u_2$, $u_3$ and $u_4$ has at least three neighbors from $(B_3 \cup B_2 \cup B_1)\setminus \{v\}$. Therefore, $\w(A) \ge 17+5 \cdot 35 + 4 \cdot 3 + 11 \cdot 2 =226 > 105\, |A|$, a contradiction.

The argumentation is similar if we suppose that a special vertex $v$ is adjacent to $u_0$ from $W_0$ and to a vertex $u_1$ from the $7$-cycle $u_1\dots u_7u_1$. Here we set $A=\{v, u_3, u_6\}$ and observe that $\w(A) \ge 17+ 8 \cdot 35 + 4 \cdot 3+ 20 \cdot 2= 349 > 105\, |A|$ that contradicts our assumption on $G_D$. \smallqed

\paragraph{Claim F.}  If $v_1$ and $v_2$ are two adjacent vertices from $W_1$, then at most one of them may have a special blue neighbor.

\Proof  Assume to the contrary that  $v_1$ is adjacent to the special vertex $u_1$, and $v_2$ is adjacent to the special vertex $u_2$. Denote the other neighbors of $u_1$ and $u_2$ by $x_1$ and $x_2$, respectively. Hence, $v_1, v_2 \in W_1$, $u_1, u_2 \in B_2$ and $x_1, x_2 \in W_0$ hold in $G_D$. Consider the set $A=\{u_1, u_2\}$ and observe that all the six vertices become red in $G_{D \cup A}$. Further, for $i=1,2$, vertex $x_i$ has at least four neighbors from $(B_2 \cup B_1)\setminus \{u_i\}$ and $v_i$ has at least three neighbors from $(B_3 \cup B_2 \cup B_1)\setminus \{u_i\}$. Thus, $\w(A) \ge 2\cdot 17 + 4\cdot 35 + 8 \cdot 3 + 6\cdot 2 = 210 = 105\, |A|$ and this contradiction proves the claim. \smallqed

\medskip

Having Claims A-F in hand, we are ready to prove that every $G_D$ (where $D$ is not a dominating set) satisfies Property 1. The last step of this proof is based on a discharging. 

\paragraph{Discharging.} First, we assign charges to the (non-red) vertices of $G_D$ so that every white vertex gets $35$, and every vertex from $B_3$, $B_2$, and $B_1$ gets $19$, $17$, and $14$, respectively. Note that the sum of the charges equals $f(G_D)$. Then, every blue vertex, except the special ones, distributes its charge equally among the white neighbors. The exact rules are the following:
\begin{itemize}
\item Every vertex from $B_3$ gives $19/3$ to each white neighbor.
\item Every non-special vertex from $B_2$ gives $17/2$ to each white neighbor.
\item Every special vertex gives $14$ to its neighbor from $W_0$, and gives $3$ to the other neighbor.
\item Every vertex from $B_1$ gives $14$ to its neighbor.
\end{itemize}
  
After the discharging, every vertex from a $P_1$-component of $G_D$ has a charge of at least $35+5\cdot 14=105$. By Claim F, every  $P_2$-component has at least four non-special blue neighbors and, therefore, its charge is at least $2\cdot 35 + 4 \cdot 3 + 4\cdot 19/3=321/3$. By Claim E, every $C_4$-component has at least $4\cdot 35 + 12 \cdot 19/3 = 216 $ and every $C_7$-component has at least $7\cdot 35+ 21\cdot 19/3=378$ as a charge. Finally, every $C_5$-component has $5\cdot 35+ 15\cdot 3= 220 $, and every $C_{10}$-component has $10\cdot 35 + 30\cdot 3 =440$ after the discharging. Let the number of $P_1$-, $P_2$-, $C_4$- $C_5$-, $C_7$-, and $C_{10}$-components of $G_D[W]$ be denoted by  $p_1$, $p_2$, $c_4$, $c_5$, $c_7$, and $c_{10}$, respectively, and let $A$ be a minimum dominating set in $G_D[W]$. Then,
$$|A| = p_1 +p_2+ 2 \, c_4 +2 \, c_5 + 3 \,  c_7 +4 \, c_{10}.$$
As $D\cup A$ is a dominating set in the graph $G$, we have $f(G_{D\cup A})=0$. Thus, $\w(A)=f(G_D)$, and the discharging shows the following lower bound:
\begin{equation*}
\begin{split}
\w(A) = f(G_D) &\ge 105 \, p_1 + \frac{321}{3} \, p_2+ 216 \, c_4 +220 \,c_5 + 378 \,c_7 +440 \, c_{10} \\
& \ge 105\, (p_1 +p_2+ 2 \, c_4 +2 \, c_5 + 3 \,  c_7 +4 \, c_{10}) = 105 \,|A|.
\end{split}
\end{equation*}
As it contradicts our assumption on $G_D$, we infer that every graph $G$ with minimum degree $5$ and every $D \subseteq V(G)$ with $f(G_D) >0$ satisfy Property 1. 
\medskip

To finish the proof of Theorem~\ref{thm:delta-5}, we first observe that $f(G_\emptyset)= 35 \,n$. Then, by Property 1, there exists a nonempty set $A_1$  such that $f(G_{A_1}) \le f(G_\emptyset) - 105\, |A_1|$.  Applying this iteratively, at the end we obtain a dominating set $D= A_1 \cup \cdots \cup A_j$ such that 
$$f(G_D)=0 \le f(G_\emptyset) - 105 |D| = 35\, n -105 |D |,$$
and we may conclude
$$\gamma(G) \le |D| \le \frac{35\, n}{105} = \frac{n}{3}.$$ 
\end{proof}

In a graph $G$, a set $X \subseteq V(G)$ is a \emph{2-packing}, if     any two different vertices from $X$ are at a distance of at least $3$.  The proof of Theorem~\ref{thm:delta-5} directly corresponds to an algorithm that outputs a dominating set of cardinality at most $n/3$. If $G$ is $5$-regular and $X$ is a $2$-packing in it, we may start the algorithmic process with choosing the vertices of $X$ one by one. Hence, we conclude the following.
\begin{corollary}  If $X$ is a $2$-packing in a $5$-regular graph $G$, then $X$ can be extended to a dominating set $D$ of cardinality at most $n/3$.
\end{corollary}

\section{Graphs of minimum degree $4$}  \label{sec:deg-4}
In this section, we apply the previous approach for graphs of minimum degree four and get a shorter alternative proof for the following theorem which was first proved by Sohn and  Xudong \cite{sohn-2009} in 2009.

\begin{theorem}
\label{thm:delta-4}
	For every graph $G$ on $n$ vertices and with minimum degree $4$, the domination number satisfies $\gamma(G) \le \frac{4n}{11}$.
\end{theorem}
\begin{proof}
Consider a graph $G$ of minimum degree $4$ and let $D$ be a subset of $V=V(G)$. Let $W$, $B$, and $R$ denote the set of white, blue, and red vertices in $G_D$. The set of blue vertices that have at least $4$ white neighbors is denoted by $B_4$ while, for $i=1,2,3$, $B_i$ stands for the set of blue vertices that have exactly $i$ white neighbors. In the proof, a residual graph $G_D$ is associated with the following value:
$$g(G_D)= 16|W|+10|B_4| +9|B_3| +8|B_2| +7|B_1|.$$
For a set $A \subseteq V\setminus D$, we use the notation
$$\w(A)=g(G_D)-g(G_{D\cup A})$$
and define the following property for $G_D$: 
 \paragraph{Property 2.} There exists a nonempty set $A \subseteq V\setminus D$ such that $\w(A) \ge 44\,|A|$.
 \medskip
 
 We now suppose for a contradiction that a residual graph $G_D$ with $\delta(G)=4$ and  $g(G_D)>0$ does not satisfy Property 2. We prove several claims for $G_D$ and then get the final contradiction via performing a discharging.
 
\paragraph{Claim G.}  $\Delta_W(W) \le 2$ and $\Delta_W(B) \le 3$ hold.

\Proof All the following cases can be excluded:
\begin{itemize}
\item \textbf{Case 1.} $\Delta_W(W) \ge 5$\\
Choose a white vertex $v$ with $d_W (v) \ge 5$ and let $A=\{v\}$. In $G_{D\cup A}$, the white vertex $v$ becomes red and its white neighbors  become blue or red. This gives $\w(A)\ge 16+ 5\cdot (16-10)=46 >44\, |A|$ which contradicts our assumption that $G_D$ does not satisfy Property 2.

\item \textbf{Case 2.} $\Delta_W(W) =4$\\
 Consider a white vertex $v$ with $d_W(v)=4$ and set $A=\{v\}$. In $G_{D\cup A}$, the vertex $v$ becomes red and its white neighbors become blue or red. Since each white neighbor $u$ had at most four white neighbors in $G_D$, $u$ may have at most three white neighbors in $G_{D\cup A}$. Therefore, $\w(A) \ge 16+ 4 \cdot (16-9) = 44\, |A|$, a contradiction. 

\item \textbf{Case 3.} $\Delta_W(W) \le 3$ and $\Delta_W(B) \ge 5$\\
 Let $v$ be a blue vertex with $d_W(v) \ge 5$ and define $A=\{v\}$ again. In $G_D$, the vertex $v$ belongs to $B_4$, while we have $v\in R$ in $G_{D\cup A}$.  Further, since $\Delta_W(W) \le 3$, each white neighbor $u$ of $v$ has at most three white neighbors in $G_D$ and  $u \in B_3 \cup B_2 \cup B_1 \cup R$ in $G_{D\cup A}$. As follows, $\w(A) \ge 10+5(16-9)=45 >44\, |A|$ that is a contradiction to our assumption. 

\item \textbf{Case 4.} $\Delta_W(W) = 3$ and $\Delta_W(B) \le 4$\\
First remark that, by the condition $\Delta_W(B) \le 4$, if a blue vertex loses $\ell$ white neighbors in a step, then $g(G_D)$ decreases by at least $\ell$. Select a white vertex $v$ with $d_W(v) =3$ and let $A=\{v\}$. In $G_{D\cup A}$, vertex $v$ becomes red and its three white neighbors become blue or red having at most $2$ white neighbors. By Observation~\ref{obs:1} $(iv)$, each of $v$ and its white neighbors has at least one blue neighbor in $G_D$. Thus, we get $\w(A) \ge 16+ 3(16-8) + 4\cdot 1= 44\, |A|$ which is a contradiction. 

\item \textbf{Case 5.} $\Delta_W(W) \le 2$ and $\Delta_W(B) = 4$\\
 Here, we choose a vertex $v$ from $B_4$ and define $A=\{v\}$. First, observe that $v$ belongs to $ B_4$ in $G_D$ and to $R$ in $G_{D \cup A}$. In $G_D$, $v$ has four white neighbors which become blue or red and belong to $ B_2 \cup B_1 \cup R$ in $G_{D\cup A}$. By Observation~\ref{obs:1} $(iv)$ and by $\Delta_W(W) \le 2$, each white neighbor has at least one blue neighbor that is different from $v$.  Therefore, $\w(A) \ge 10+ 4(16-8)+4\cdot 1=46>44\, |A|$ that is a contradiction again. This finishes the proof of the claim.  \smallqed
\end{itemize}

In the continuation, we suppose that $\Delta_W(W) \le 2$ and $\Delta_W(B) \le 3$ hold in the counterexample $G_D$ and, therefore, the graph $G_D[W]$, which is induced by the white vertices of $G_D$, consists of components which are paths and cycles. We prove some further properties for $G_D$.

\paragraph{Claim H.}  In $G_D[W]$, each component is a path $P_1$, $P_2$ or a cycle $C_4$ or $C_7$.

\Proof Assume that there is a path component $P_j\colon v_1\dots v_j$ of order $j \ge 3$ in $G_D[W]$. We set $A=\{v_2\}$ and observe that both $v_1$ and $v_2$ become red and $v_3$ belongs to $B_1 \cup R$ in $G_{D\cup A}$. This contributes to $\w(A)$ by at least $2\cdot 16 + (16- 7)$. By Observation~\ref{obs:1} $(iv)$, $v_1$, $v_2$, and $v_3$, respectively, have at least $3$, $2$, $2$ blue neighbors in $G_D$. The decrease in their white-degrees contributes to $\w(A)$ by at least $7 \cdot 1$. Then, we get $\w(A) \ge 32+9+7=48 >44\, |A|$, a contradiction.

Now, assume that a cycle $C_{3k}\colon v_1 \dots v_{3k}v_1$ exists in $G_D[W]$ and set $A=\{v_3, v_6, \dots , v_{3k}\}$. In $G_{D \cup A}$, all the $3k$ vertices of the cycle are recolored  red and, by Observation~\ref{obs:1} $(iv)$, the sum of the white-degrees of the blue vertices decreases by at least $2\cdot 3k$. Consequently, we get the contradiction $w(A) \ge 16\cdot 3k+ 6k = 54k > 44\, |A|$. A similar argumentation can be given if the cycle is $C_{3k+2}\colon v_1 \dots v_{3k+2}v_1$, where $k \ge 1$, and $A=\{v_3, v_6, \dots , v_{3k}, v_{3k+2}\}$. Here, $|A|=k+1$ and we get $\w(A) \ge 16\cdot (3k+2) + 2\cdot (3k+2) = 54k +36 > 44k +44 =44\, |A|$ that is a contradiction. For the case when the cycle is of order $3k+1$, we suppose $k \ge 3$ and obtain a contradiction as follows. Let $C_{3k+1}\colon v_1 \dots v_{3k+1}v_1$ and let $A$ be the $(k+1)$-element dominating set $\{v_3, v_6, \dots , v_{3k}, v_{3k+1}\}$. We get $\w(A) \ge 16\cdot (3k+1) + 2\cdot (3k+1) = 54k +18 > 44k+44 =44\, |A|$ since $k \ge 3$ is supposed. This finishes the proof of Claim H. \smallqed

\paragraph{Claim I.}  No vertex from $B_3$ is adjacent to any vertices from $W_0$ in $G_D$.

\Proof Assume for a contradiction that a vertex $v\in B_3$ has a neighbor $u_0$ from $W_0$. Let $A=\{v\}$ and denote by $u_1$ and $u_2$ the further two white neighbors of $v$. In $G_{D\cup A}$, $v,u_0 \in R$ and $u_1,u_2 \in B_2 \cup B_1 \cup R$. This change contributes to $\w(A)$ by at least $9+16+2(16-8)=41$. By Observation~\ref{obs:1} $(iv)$, the neighbors $u_0$, $u_1$ and $u_2$ have, respectively, at least $3$, $1$, $1$ blue neighbors which are different from $v$. Therefore, $\w(A) \ge 41+ 5 \cdot 1 = 46 > 44\, |A|$ should be true but this contradicts our assumption on $G_D$. \smallqed
\medskip
 
 As follows, the vertices from $W_0$ may be adjacent only to some vertices from $B_2 \cup B_1$. We call a vertex from $B_2$ \emph{special}, if it is adjacent to a vertex from $W_0$. 

\paragraph{Claim J.}  No special vertex is adjacent to two vertices from $W_0$. 

\Proof Suppose that a vertex $v \in B_2$ is adjacent to two vertices, say $u_1$ and $u_2$ from $W_0$. We set $A=\{v\}$ and observe that all the three vertices $v$, $u_1$ and $u_2$ are red in $G_{D\cup A}$. By Observation~\ref{obs:1} $(iv)$, each of $u_1$ and $u_2$ has at least three blue neighbors different from $v$. This yields $\w(A) \ge 8+2\cdot 16+ 6 \cdot 1= 46>44\, |A|$ that contradicts our assumption on $G_D$. \smallqed 
  
\paragraph{Claim K.}  No special vertex is adjacent to a vertex from a $C_4$ or $C_7$. 

\Proof If a special vertex $v$ is adjacent to  a vertex $u_0$ from $W_0$ and to a vertex $u_1$ from a $4$-cycle component $C_4\colon u_1u_2u_3u_4u_1$ of $G_D[W]$, then we set $A=\{v, u_3\}$ and observe that $v$, $u_0$, $u_1$, $u_2$, $u_3$ and $u_4$ turn red in $G_{D\cup A}$.  In $G_D$, the vertices $u_0$, $u_1$, $u_2$, $u_3$ and $u_4$, respectively, have at least $3$, $1$, $2$, $2$, $2$  neighbors from $(B_3 \cup B_2 \cup B_1)\setminus \{v\}$. Thus, $\w(A) \ge 8+5 \cdot 16 + 10 \cdot 1 =98 > 44\, |A|$, a contradiction. Similarly, if we suppose that a special vertex $v$ is adjacent to $u_0$ from $W_0$ and to a vertex $u_1$ from the $7$-cycle $u_1\dots u_7u_1$, we set $A=\{v, u_3, u_6\}$ and conclude that $\w(A) \ge 8+ 8 \cdot 16 + 16 \cdot 1= 152 > 44\, |A|$ that contradicts our assumption on $G_D$. \smallqed

\paragraph{Claim L.}  If $v_1$ and $v_2$ are two adjacent vertices from $W_1$, then at most one of them may have a special blue neighbor.

\Proof  Assume to the contrary that  $v_1u_1, v_2u_2 \in E(G)$ such that $u_1$, and $u_2$ are special vertices in $G_D$, and let $x_1$ and $x_2$ be the further white neighbors of $u_1$ and $u_2$. Hence, we have $v_1, v_2 \in W_1$, $u_1, u_2 \in B_2$, and $x_1, x_2 \in W_0$ in $G_D$. Consider the set $A=\{u_1, u_2\}$ and observe that all the six vertices $v_1$, $v_2$, $u_1$, $u_2$, $x_1$, $x_2$ become red in $G_{D \cup A}$. For $i=1,2$, by Claim I and  Observation~\ref{obs:1} $(iv)$, the vertex $x_i$ has at least three neighbors from $(B_2 \cup B_1)\setminus \{v\}$ and $v_i$ has at least two neighbors from $(B_3 \cup B_2 \cup B_1)\setminus \{v\}$. This implies the contradiction $\w(A) \ge 2\cdot 8 + 4\cdot 16 + 10 \cdot 1 = 90 >44\, |A|$.  \smallqed

\paragraph{Discharging.} Applying Claims G--L, we now perform a discharging and prove that $G_D$ satisfies Property 1. We assign charges to the (non-red) vertices of $G_D$ so that every white vertex gets $16$, and every vertex from $B_3$, $B_2$, and $B_1$ gets $9$, $8$, and $7$, respectively. We remark that the sum of these charges equals $g(G_D)$. Then, every blue vertex, except the special ones, distributes its charge equally among the white neighbors as follows:
\begin{itemize}
\item Every vertex from $B_3$ gives $3$ to each white neighbor.
\item Every non-special vertex from $B_2$ gives $4$ to each white neighbor.
\item Every special vertex gives $7$ to its neighbor from $W_0$, and gives $1$ to the other neighbor.
\item Every vertex from $B_1$ gives $7$ to its neighbor.
\end{itemize}
  
After the discharging, every vertex from a $P_1$-component of $G_D[W]$ has a charge of at least $16+4\cdot 7=44$. By Claim L, every $P_2$-component has at least three non-special blue neighbors and, therefore, its charge is at least $2\cdot 16 + 3 \cdot 1 + 3\cdot 3=44$. By Claim K, every $C_4$-component has at least $4\cdot 16 + 8 \cdot 3 = 88$ and every $C_7$-component has at least $7\cdot 16+ 14\cdot 3=154$ as a charge. Let the number of $P_1$-, $P_2$-, $C_4$-, and $C_7$-components of $G[W]$  be denoted by  $p_1$, $p_2$, $c_4$, and $c_7$, respectively, and let $A$ be a minimum dominating set in $G[W]$. Then,
$$|A| = p_1 +p_2+ 2 \, c_4  + 3 \,  c_7. $$
As $D\cup A$ is a dominating set in the graph $G$, we have $g(G_{D\cup A})=0$. Thus, $\w(A)=g(G_D)$, and the discharging proves the following lower bound:
\begin{equation*}
\begin{split}
\w(A) = g(G_D) &\ge 44 \, p_1 + 44 \, p_2+ 88 \, c_4 +154 \,c_7 \\
& \ge 44\, (p_1 +p_2+ 2 \, c_4 +3 \,  c_7) = 44 \,|A|.
\end{split}
\end{equation*}
As it contradicts our assumption on $G_D$, we infer that every graph $G$ with minimum degree $4$ and every $D \subseteq V(G)$ with $g(G_D) >0$ satisfy Property 2. 

To prove Theorem~\ref{thm:delta-4}, we observe that $g(G_\emptyset)= 16 \,n$ and, by Property 2, there exists a set $A_1$  such that $g(G_{A_1}) \le g(G_\emptyset) - 44\, |A_1|$.  As $G_{A_1}$ also satisfies Property 2, we may continue the process if $g(G_{A_1}) >0$, and at the end we obtain a dominating set $D= A_1 \cup \cdots \cup A_j$ such that 
$$g(G_D)=0 \le g(G_\emptyset) - 44\, |D| = 16\, n -44\, |D |.$$
Consequently,
$$\gamma(G) \le |D| \le \frac{16\, n}{44} = \frac{4}{11}\,n$$ 
holds for every graph $G$ of minimum degree $4$.
\end{proof}

\section{Concluding remarks}
Theorem~\ref{thm:delta-5} shows that $\gamma(G) \le n/3$ holds for every graph with minimum degree at least $5$. However, I do not believe that this upper bound is tight over the class of graphs with $\delta(G)\ge 5$. Examples with $\gamma/n >1/4$ can possibly be found among larger graphs via computer search or large constructions, but it seems that $\delta(G) \ge 5$ and $n \le 12$ together implies $\gamma(G)\le n/4$ that is quite far from the proved $n/3$-upper bound.

Unfortunately, Theorem~\ref{thm:delta-4} does not seem sharp either. However, here we have $4$-regular examples where the quotient $\gamma/n$ equals $1/3$ that is relatively close to the proved upper bound $4/11$. The smallest such $4$-regular graph is $G=K_6-M$ that is obtained from the complete graph $K_6$ by the deletion of a perfect matching. Then, we have  $\gamma(G)=2=n/3$. One may guess that this is the sharp upper bound for graphs of minimum degree $4$ or, at least, it is true under the following stronger condition:
\begin{conj}
	There exists a constant $n_0$ such that for every connected $4$-regular graph $G$ of order $n >n_0$, we have $\gamma(G) \le \frac{n}{3}$.
\end{conj}

	%

\end{document}